\documentclass[11pt, twoside]{article}
\usepackage{amssymb}
\usepackage{mathrsfs}
\usepackage{bbm}
\usepackage{cite}
\pdfoutput=1
\usepackage{amsfonts}
\usepackage{amsmath,amssymb}
\usepackage{multicol}

\usepackage{color}

\newtheorem{theorem}{Theorem}[section]
\newtheorem{lemma}[theorem]{Lemma}
\newtheorem{corollary}[theorem]{Corollary}

\newtheorem{definition}[theorem]{Definition}
\newtheorem{remark}[theorem]{Remark}

\numberwithin{equation}{section}
\newenvironment{proof}[1][Proof]{\noindent\textbf{#1.} }{\hfill $\Box$}

\allowdisplaybreaks \makeatletter
\begin{document}
\author{Rui Li and Shuangping Tao \thanks{%
Corresponding author (taosp@nwnu.edu.cn).}  \\
%EndAName
\\College of Mathematics and Statistics, Northwest Normal University,\\
Lanzhou, Gansu 730070, China\\
{\small\it E-mail: lirui18894039032@163.com~~~taosp@nwnu.edu.cn}}
\date{}

\title{\vspace{-4cm}\textbf{Sobolev regularity for a class of local fractional new maximal operators }} \maketitle

$\mathbf{Abstract:}$ {This paper is devoted to studying the regularity properties for the new maximal operator $M_{\varphi}$ and the fractional new maximal operator $M_{\varphi,\beta}$ in the local case. Some new pointwise gradient estimates of $M_{\varphi,\Omega}$ and $M_{\varphi,\beta,\Omega}$ are given. Moreover, the boundedness of $M_{\varphi,\Omega}$ and $M_{\varphi,\beta,\Omega}$ on Sobolev space is established. As applications, we also obtain the bounds of the above operators on Sobolev space with zero boundary values.}
\\

$\mathbf{Keywords:}$ {Local fractional new maximal operator, regularity, Sobolev boundedness, partial derivative, pointwise gradient inequality}
\\


\section{Introduction and main results}

\noindent

It is well known that fractional type maximal operators play important roles in harmonic analysis, partial differential equations, and other fields. In 2012, in order to study weighted $L^{p}$ inequalities for the pseudo-differential operators with smooth symbols and their commutators, Tang introduced a class of new maximal operator $M_{\varphi}$ in \cite{T6} as follows,
$$
M_{\varphi}f(x)={\sup\limits_{x\in Q}\frac{1}{\varphi(|Q|)|Q|}\int_{Q}|f(y)|\mathrm{d}y},
$$
where $\varphi(t)=(1+t)^{\gamma}$ for ${t}\geq{0}$ and ${\gamma}\geq{0}$, and the supremum is taken over all the cubes $Q$. Subsequently, the following fractional new maximal operator $M_{\varphi,\beta}$ was given in \cite{HC7},
$$
M_{\varphi,\beta}f(x)={\sup\limits_{x\in Q}\frac{1}{(\varphi(|Q|)|Q|)^{1-\frac{{\beta}}{n}}}\int_{Q}|f(y)|\mathrm{d}y},
$$
where $0\leq{\beta}<{n}$ and the supremum is taken over all the cubes $Q$. As $\beta=0$,  we denote $M_{\varphi,0}$  by $M_{\varphi}$ for simplicity, which is just the new maximal operator. Clearly, if $\varphi(|Q|)=1$, then $M_{\varphi}$ will reduce to the Hardy-Littlewood maximal operator $M$.

In this paper, we consider the regularity properties of $M_{\varphi,\beta}$ in the local case, but with the cubes $Q$ replaced by the balls $B$ because the local fractional new maximal operator over balls $B$ has better smoothing properties than the local fractional new maximal operator over cubes $Q$. Next, we give the definition of the local fractional new maximal operator.

\begin{definition}
\rm{Let $0\leq{\beta}<{n}$, $0\leq{\gamma}<\infty$, and $\Omega$ be a subdomain on $\mathbb{R}^{n}$. Then the local fractional new maximal operator is defined by
$$
{M}_{\varphi,\beta,\Omega}{f}(x)=\sup\limits_{0<{r}<{\mathrm{dist}(x,\Omega^{c})}}\frac{r^{\beta}(1+r^{n})^{\frac{\gamma\beta}{n}}}
{\varphi(|B(x,r)|)|B(x,r)|}\int_{B(x,r)}|f(y)|\mathrm{d}y,
$$
where $\Omega^{c}=\mathbb{R}^{n}\backslash{\Omega}$, $B=B(x,r)$ is denoted as a ball, $x$ and $r$ are its center and radius.}
\end{definition}

\begin{remark}
\rm{As $\gamma=0$, ${M}_{\varphi,\beta,\Omega}{f}(x)=:{M}_{\beta,\Omega}{f}(x)$ is called the local fractional maximal operator. Specially, as $\beta=0$, ${M}_{\varphi,0,\Omega}$ and ${M}_{0,\Omega}$ are coincide with the local new maximal operator and the local Hardy-Littlewood maximal operator, denoted by $M_{\varphi,\Omega}$ and ${M}_{\Omega}$, respectively. Obviously, if $\Omega=\mathbb{R}^{n}$, then ${M}_{\beta,\Omega}{f}(x)=:{M}_{\beta}{f}(x)$ is just the fractional maximal operator and the supremum is taken over all $r>0$.}
\end{remark}

The regularity properties of maximal type operators have been an active subject of current articles. This program was originally introduced in the global case $\Omega=\mathbb{R}^{n}$ by Kinnunen \cite{K1} who showed that the Hardy-Littlewood maximal operator $M$ is bounded on Sobolev space $W^{1,p}(\mathbb{R}^{n})$ for $1<p\leq{\infty}$. Where Sobolev space $W^{1,p}(\mathbb{R}^{n})$ is defined by
\begin{align}\label{001}
W^{1,p}(\mathbb{R}^{n})=\{f:\mathbb{R}^{n}\rightarrow{\mathbb{R}}: \|f\|_{W^{1,p}(\mathbb{R}^{n})}=\|f\|_{L^{p}(\mathbb{R}^{n})}
+\|\nabla{f}\|_{L^{p}(\mathbb{R}^{n})}<\infty\},
\end{align}
here $\nabla{f}=(D_{1}f,D_{2}f,\cdots,D_{n}f)$ denotes the weak gradient of $f$, $D_{i}f(i=1,2,\cdots,n)$ denote the weak partial derivatives of $f$. The pointwise gradient estimate of $M$ was also obtained in \cite{K1} as follows,
$$
|\nabla{{M}{f}}(x)|\leq{{M}{|\nabla{f}|}(x)},
$$
for almost every $x\in{\mathbb{R}^{n}}$. Since then, the regularity properties of $M$ have been studied extensively, for example see \cite{AP1, HO11, K11, L11, T11}. In 2003, the work of Kinnunen was extended to the fractional setting in \cite{KS},
$$
|\nabla{{M}_{\beta}{f}}(x)|\leq{{M}_{\beta}{|\nabla{f}|}(x)}~~~\mathrm{and}~~~ |\nabla{{M}_{\beta}{f}}(x)|\leq{C{M}_{\beta-1}{f}(x)},
$$
where $f\in{W^{1,p}(\mathbb{R}^{n})}$ and $f\in{L^{p}(\mathbb{R}^{n})}$, respectively. As a generalization for the regularity of maximal operators in the global setting, it is more difficult to discuss the pointwise gradient estimates and Sobolev  boundedness of $M$ and $M_{\beta}$ in the local case because the local maximal operators lack commutativity with translations. The first attempt was made by Kinnunen and Lindqvist in \cite{KL6}. They proved that if $f\in{W^{1,p}(\Omega)}(1<p<\infty)$, then $M_{\Omega}$ is bounded on Sobolev space $W^{1,p}(\Omega)$ by establishing the following pointwise gradient estimate
$$
|\nabla{{M}_{\Omega}{f}}(x)|\leq{2{M}_{\Omega}{|\nabla{f}|}(x)},
$$
for almost every $x\in{\Omega}$. Where the definition of $W^{1,p}(\Omega)$ is the same as \eqref{001}, but using $\Omega$ instead of $\mathbb{R}^{n}$. For the local fractional case, it wasn't until 2015 that Heikkinen et al. obtained the following inequalities in \cite{HKK6},
\begin{align}\label{101}
|\nabla{{M}_{\beta,\Omega}{f}}(x)|
\leq{\beta{M}_{\beta-1,\Omega}{f}(x)+2{M}_{\beta,\Omega}{|\nabla{f}|}(x)},
\end{align}
and there exists a positive constant $C$ depends only on $n$, such that
\begin{align}\label{102}
|\nabla{{M}_{\beta,\Omega}{f}}(x)|
\leq{C({M}_{\beta-1,\Omega}{f}(x)+{S}_{\beta-1,\Omega}{f}(x))},
\end{align}
for almost every $x\in{\Omega}$. Noting that the inequality \eqref{101} cannot be obtained as $\Omega$ has infinite measure. Thus, Heikkinen et al. \cite{HKK6} added the condition $|\Omega|<\infty$ in the inequality \eqref{101}. Another feature is that the inequalities \eqref{101} and \eqref{102} contain the extra terms ${M}_{\beta-1,\Omega}$ and ${S}_{\beta-1,\Omega}$ cannot be dismissed. Here ${S}_{\beta,\Omega}$ is the local spherical fractional maximal operator, defined by
$$
{S}_{\beta,\Omega}{f}(x)=\sup\limits_{0<{r}<{\mathrm{dist}(x,\Omega^{c})}}\frac{r^{\beta}}
{|\partial{B(x,r)}|}\int_{\partial{B(x,r)}}|f(y)|\mathrm{d}\mathcal{H}^{n-1}(y),
$$
where $\mathrm{d}\mathcal{H}^{n-1}$ denotes the normalized (n-1)-dimensional Hausdorff measure, $|\partial{B(x,r)|}=n\varpi_{n}r^{n-1}$, and $\varpi_{n}$ is the volume of the unit ball on $\mathbb{R}^{n}$. The further development about the regularity of maximal operators, we can see \cite{AA, B1, HLX2, KI, LT, LW, LWX1, LWX2, LXY, LZ, S, ZL} and so on.

Inspired by the works above, especially the results in \cite{HKK6, HLX2}, the  main purpose of this paper is to consider the pointwise estimates for the weak gradient of ${M}_{\varphi,\beta,\Omega}$ with $W^{1,p}(\Omega)$ and $L^{p}(\Omega)$ functions and the boundedness of ${M}_{\varphi,\beta,\Omega}$ on Sobolev space and Sobolev space with zero boundary values. We expect that the inequalities \eqref{101} and \eqref{102} would also hold for ${M}_{\varphi,\beta,\Omega}$, but this is not true as such. To achieve our desired results, we have to consider two cases $0<r<1$ and $1\leq{r}<\mathrm{dist}(x,\Omega^{c})$ and obtain different conclusions under different conditions. What should be stressed is that some results in \cite{HKK6, HLX2} can be concluded by our theorems and corollary as $\gamma=0$. Thus, we extend some results in \cite{HKK6, HLX2}.

Before stating the main results, we define the following local spherical fractional new maximal operator,
$$
{S}_{\varphi,\beta,\Omega}{f}(x)=\sup\limits_{0<{r}<{\mathrm{dist}(x,\Omega^{c})}}\frac{r^{\beta}(1+r^{n})^{\frac{\gamma\beta}{n}}}
{\varphi(|\partial{B(x,r)|})|\partial{B(x,r)}|}\int_{\partial{B(x,r)}}|f(y)|\mathrm{d}\mathcal{H}^{n-1}(y),
$$
where $0\leq{\beta}<{n}$ and $0\leq{\gamma}<\infty$.

Our main results are formulated as follows. We first give the pointwise gradient estimates for ${M}_{\varphi,\beta,\Omega}$ with $W^{1,p}(\Omega)$ function.

\begin{theorem}\label{thm0}
Let $f\in{W^{1,p}(\Omega)}$ with $1<p<n$, $\frac{1}{q}=\frac{1}{p}-\frac{\beta}{n}$, and $1\leq{\beta}<\frac{n}{p}$. Then

(i)~~If $0<r<1$, $0\leq\gamma\leq{n}$, and one of the following conditions holds:

~~~(a)~~$\Omega$ admits a p-Sobolev embedding;

~~~(b)~~$|\Omega|<\infty$ and $\frac{1}{q}=\frac{1}{p}-\frac{\beta-1}{n}$.\\
Then ${{M}_{\varphi,\beta,\Omega}{f}}\in{W^{1,q}(\Omega)}$. Moreover, for almost every $x\in{\Omega}$, it holds that
$$
|\nabla{{M}_{\varphi,\beta,\Omega}{f}}(x)|
\leq{(\gamma{n}+2\beta){M}_{\varphi,\beta-1,\Omega}{f}(x)
+2{M}_{\varphi,\beta,\Omega}{|\nabla{f}|}(x)}.
$$

(ii)~~If $1\leq{r}<\mathrm{dist}(x,\Omega^{c})$, then ${{M}_{\varphi,\beta,\Omega}{f}}\in{W^{1,q}(\Omega)}$. Moreover, for almost every $x\in{\Omega}$, we have
$$
|\nabla{{M}_{\varphi,\beta,\Omega}{f}}(x)|
\leq{(\gamma{n}+\beta){M}_{\varphi,\beta,\Omega}{f}(x)+2{M}_{\varphi,\beta,\Omega}{|\nabla{f}|}(x)}.
$$
\end{theorem}

\begin{remark}
\rm{Note that the condition $\Omega$ admits a p-Sobolev embedding or the condition $|\Omega|<\infty$ in $(ii)$ of Theorem \ref{thm0} is not needed. That is, in order to achieve our results, we only need to add the condition $\Omega$ admits a p-Sobolev embedding or the condition $|\Omega|<\infty$ as $0<r<1$. In addition, the condition $(a)$ is better than condition $(b)$ because $\Omega$ may have unbounded measure and $q$ has better value as the condition $(a)$ holds.}
\end{remark}

In addition, the pointwise gradient estimates for ${M}_{\varphi,\beta,\Omega}$ with $L^{p}(\Omega)$ function are obtained as follows.

\begin{theorem}\label{thm2}
Let $f\in{L^{p}(\Omega)}$, $n\geq{2}$, and $p>{\frac{n}{n-1}}$. Then

(i)~~If $0<r<1$ and $\frac{1}{q}=\frac{1}{p}-\frac{\beta-1}{n}$ with $1\leq{\beta}<\min\{{n-\frac{2n}{p(n-1)}},\frac{n-1}{p}\}+1$, then $|\nabla{{M}_{\varphi,\beta,\Omega}{f}}|\in{L^{q}(\Omega)}$. Moreover, there is a constant $C>0$, depending on $\gamma$ and $n$, such that for almost every $x\in{\Omega}$, there holds
$$
|\nabla{{M}_{\varphi,\beta,\Omega}{f}}(x)|
\leq{C({M}_{\varphi,\beta-1,\Omega}{f}(x)+{S}_{\varphi,\beta-1,\Omega}{f}(x))}.
$$

(ii)~~If $1\leq{r}<\mathrm{dist}(x,\Omega^{c})$ and $\frac{1}{q}=\frac{1}{p}-\frac{\beta}{n}$ with $1\leq{\beta}<\min\{{n-\frac{2n}{p(n-1)}},\frac{n-1}{p}\}$, then $|\nabla{{M}_{\varphi,\beta,\Omega}{f}}|\in{L^{q}(\Omega)}$. Moreover, for almost every $x\in{\Omega}$, there exists a positive constant $C$ which depends on $\gamma$ and $n$, such that
$$
|\nabla{{M}_{\varphi,\beta,\Omega}{f}}(x)|
\leq{C({M}_{\varphi,\beta,\Omega}{f}(x)+{S}_{\varphi,\beta,\Omega}{f}(x))}.
$$
\end{theorem}

Next, we consider the boundedness of ${M}_{\varphi,\beta,\Omega}$ on Sobolev space.

\begin{theorem}\label{thm3}
Suppose that $1<p<n$, $1\leq{\beta}<\frac{n}{p}$, $0\leq\gamma\leq{n}$, and $\frac{1}{q}=\frac{1}{p}-\frac{\beta}{n}$. If $f\in{W^{1,p}(\Omega)}$ and $\Omega$ admits a p-Sobolev embedding, then there is a constant $C>0$ such that
$$
\|{M}_{\varphi,\beta,\Omega}{f}\|_{W^{1,q}(\Omega)}
\leq{C\|f\|_{W^{1,p}(\Omega)}}.
$$
\end{theorem}

\begin{theorem}\label{thm4}
Let $1<p<n$, $1\leq{\beta}<\frac{n}{p}$, $\frac{1}{q}=\frac{1}{p}-\frac{\beta-1}{n}$, and $0\leq\gamma\leq{n}$. If $f\in{W^{1,p}(\Omega)}$ and $|\Omega|<\infty$, then there exists a positive constant $C$ such that
$$
\|{M}_{\varphi,\beta,\Omega}{f}\|_{W^{1,q}(\Omega)}
\leq{C\|f\|_{W^{1,p}(\Omega)}}.
$$
\end{theorem}

\begin{theorem}\label{thm5}
Assume that $n\geq{2}$, $p>{\frac{n}{n-1}}$, $1\leq{\beta}<\min\{{n-\frac{2n}{p(n-1)}},\frac{n-1}{p}\}$, and $\frac{1}{q}=\frac{1}{p}-\frac{\beta-1}{n}$. If $f\in{L^{p}(\Omega)}$ and $|\Omega|<\infty$, then there exists a constant $C>0$ such that
$$
\|{M}_{\varphi,\beta,\Omega}{f}\|_{W^{1,q}(\Omega)}
\leq{C\|f\|_{L^{p}(\Omega)}}.
$$
\end{theorem}

Now, the boundedness for ${M}_{\varphi,\beta,\Omega}$ on Sobolev space with zero boundary values will be established, which are the applications of Theorem \ref{thm4} and Theorem \ref{thm5}. Recall that Sobolev space with zero boundary values $W_{0}^{1,p}(\Omega)(1\leq{p}<\infty)$ is defined as the completion of $\mathcal{C}_{0}^{\infty}(\Omega)$ with respect to the Sobolev norm. If $f\in{W_{0}^{1,p}(\Omega)}$ with $1<{p}<\infty$, Kinnunen and Lindqvist proved that ${M}_{\Omega}{f}\in{W_{0}^{1,p}(\Omega)}$ in \cite{KL6}. In 2015, Heikkinen et al. \cite{HKK6} showed that if $f\in{L^{p}(\Omega)}$, $p>{\frac{n}{n-1}}$, $n\geq{2}$, $1\leq{\beta}<\frac{n}{p}$, $\frac{1}{q}=\frac{1}{p}-\frac{\beta-1}{n}$, and $|\Omega|<\infty$, then ${M}_{\beta,\Omega}{f}\in{W_{0}^{1,q}(\Omega)}$.

\begin{theorem}\label{thm6}
Let $\frac{1}{q}=\frac{1}{p}-\frac{\beta-1}{n}$ and $|\Omega|<\infty$. Assume that either of the following conditions holds:

(i)~~If $f\in{W^{1,p}(\Omega)}$ with $1<p<n$, $1\leq{\beta}<\frac{n}{p}$, and $0\leq\gamma\leq{n}$.

(ii)~~If $f\in{L^{p}(\Omega)}$, $n\geq{2}$, $p>{\frac{n}{n-1}}$, and $1\leq{\beta}<\min\{{n-\frac{2n}{p(n-1)}},\frac{n-1}{p}\}$.\\
Then, it holds that ${M}_{\varphi,\beta,\Omega}{f}\in{W_{0}^{1,q}(\Omega)}$.
\end{theorem}

Finally, we give the regularity properties for $M_{\varphi,\Omega}$.

\begin{corollary}\label{coro7}
Let $f\in{W^{1,p}(\Omega)}$ and $1<p<\infty$. Then

(i)~~${M}_{\varphi,\Omega}{f}\in{W^{1,p}(\Omega)}$. Moreover, for almost every $x\in{\Omega}$, there holds
$$
|\nabla{{M}_{\varphi,\Omega}{f}}(x)|
\leq{\gamma{n}{M}_{\varphi,\Omega}{f}(x)
+2{M}_{\varphi,\Omega}{|\nabla{f}|}(x)}.
$$

(ii)~~The map ${M}_{\varphi,\Omega}: W^{1,p}(\Omega)\rightarrow{W^{1,p}(\Omega)}$ is bounded. That is, there exists a positive constant $C$ such that
$$
\|{M}_{\varphi,\Omega}{f}\|_{W^{1,p}(\Omega)}
\leq{C\|f\|_{W^{1,p}(\Omega)}}.
$$

(iii)~~If $f\in{W_{0}^{1,p}(\Omega)}$, then ${M}_{\varphi,\Omega}{f}\in{W_{0}^{1,p}(\Omega)}$.
\end{corollary}

The article is organized as follows: Section 2 is devoted to giving some major lemmas. The proofs of Theorems \ref{thm0} and \ref{thm2} will be given by using key lemmas in Section 3. In Section 4, we show the proofs of Theorems \ref{thm3}-\ref{thm5}. Combining Theorem \ref{thm4} and Theorem \ref{thm5}, the proof of Theorem \ref{thm6} shall be given in Section 5. In Section 6, we will give the proof of Corollary \ref{coro7} by establishing some lemmas for ${M}_{\varphi,\Omega}$.

\section{Some preliminaries}

\noindent

We need the following lemmas, which are vital to the proofs of our theorems.

\begin{lemma}\label{lem2.1}
Let $p>{1}$, $0<{\beta}<{\frac{n}{p}}$, and $\frac{1}{q}=\frac{1}{p}-\frac{\beta}{n}$. Then for any $f\in{L^{p}(\Omega)}$, there exists a positive constant $C$ such that
$$
\|{M}_{\varphi,\beta,\Omega}{f}\|_{L^{q}(\Omega)}\leq{C\|{f}\|_{L^{p}(\Omega)}}.
$$
\end{lemma}

\begin{proof}
Applying \cite[Theorem 2.1]{HKK6} and the fact that ${M}_{\varphi,\beta,\Omega}{f}(x)$ can be controlled by ${M}_{\beta,\Omega}{f}(x)$ for each $x\in\Omega$, then Lemma \ref{lem2.1} is not difficult to obtain. The details are omitted here.
\end{proof}

Also, we can get the following lemma by using \cite[Theorem 2.2]{HKK6}.

\begin{lemma}\label{lem2.2}
Let $n\geq{2}$, $p>{\frac{n}{n-1}}$, $0\leq{\beta}<\min\{{n-\frac{2n}{p(n-1)}},\frac{n-1}{p}\}$, and $\frac{1}{q}=\frac{1}{p}-\frac{\beta}{n}$. Then for any $f\in{L^{p}(\Omega)}$, there is a constant $C>{0}$ such that
$$
\|{S}_{\varphi,\beta,\Omega}{f}\|_{L^{q}(\Omega)}\leq{C\|{f}\|_{L^{p}(\Omega)}}.
$$
\end{lemma}

\begin{lemma}\label{lem2.3}{\textsuperscript{\cite{HKK6, KL6}}}
Let $1\leq{p}\leq{\infty}$. If for $k=1,2,\ldots$, $g_{k}\leq{f_{k}}$ almost everywhere on $\Omega$ and $g_{k}\rightarrow{g}$, $f_{k}\rightarrow{f}$ weakly on $L^{p}(\Omega)$ as $k\rightarrow{\infty}$, then $g\leq{f}$ almost everywhere on $\Omega$.
\end{lemma}

\section{Proofs of Theorems \ref{thm0} and \ref{thm2}}

\noindent

Let $0<{l}<1$ and $0\leq{\beta}<{n}$. We define the fractional new average operator by
$$
K_{\varphi,\beta}^{l}f(x)=\frac{(l\sigma(x))^{\beta}(1+(l\sigma(x))^{n})^{\frac{\gamma\beta}{n}}}
{\varphi(|B(x,l\sigma(x))|)|B(x,l\sigma(x))|}\int_{B(x,l\sigma(x))}f(y)\mathrm{d}y,
$$
where $\sigma(x)=\mathrm{dist}(x,\Omega^{c})$ is a Lipschitz function. It follows from Rademacher's theorem that the function $\sigma(x)$ is differentiable almost everywhere on $\Omega$. Moreover, $|\nabla\sigma(x)|=1$ for almost every $x\in{\Omega}$. The following estimates for the weak gradient of $K_{\varphi,\beta}^{l}$ will be used.

\begin{lemma}\label{lem3.1}
Suppose that $f\in{W^{1,p}(\Omega)}$ with $1<p<{n}$, $0<l<1$, $\frac{1}{q}=\frac{1}{p}-\frac{\beta}{n}$, and $1\leq{\beta}<\frac{n}{p}$. Then

(i)~~If $0<l\sigma(x)<1$, $0\leq\gamma\leq{n}$, and one of the following conditions holds:

~~~(a)~~$\Omega$ admits a p-Sobolev embedding;

~~~(b)~~$|\Omega|<\infty$ and $\frac{1}{q}=\frac{1}{p}-\frac{\beta-1}{n}$.\\
Then $|\nabla{K_{\varphi,\beta}^{l}f}|\in{L^{q}(\Omega)}$. Moreover, for almost every $x\in{\Omega}$, there holds
$$
|\nabla{K_{\varphi,\beta}^{l}f}(x)|
\leq{(\gamma{n}+2\beta){M}_{\varphi,\beta-1,\Omega}{f}(x)+2{M}_{\varphi,\beta,\Omega}{|\nabla{f}|}(x)}.
$$

(ii)~~If $l\sigma(x)\geq1$, then $|\nabla{K_{\varphi,\beta}^{l}f}|\in{L^{q}(\Omega)}$. Moreover, for almost every $x\in{\Omega}$, we have
$$
|\nabla{K_{\varphi,\beta}^{l}f}(x)|
\leq{(\gamma{n}+\beta){M}_{\varphi,\beta,\Omega}{f}(x)+2{M}_{\varphi,\beta,\Omega}{|\nabla{f}|}(x)}.
$$
\end{lemma}

\begin{proof}
We first suppose that $f\in{W^{1,p}(\Omega)\cap{\mathcal{C}^{\infty}(\Omega)}}$. Then for almost every $x\in{\Omega}$ and $1\leq{i}\leq{n}$, from the Leibnitz rule there holds
\begin{eqnarray*}
\begin{split}
D_{i}K_{\varphi,\beta}^{l}f(x)
&=D_{i}\left(\frac{(l\sigma(x))^{\beta}(1+(l\sigma(x))^{n})^{\frac{\gamma\beta}{n}}}
{(1+\varpi_{n}(l\sigma(x))^{n})^{\gamma}\varpi_{n}(l\sigma(x))^{n}}\right)\int_{B(x,l\sigma(x))}f(y)\mathrm{d}y\\
&~~~+\frac{(l\sigma(x))^{\beta}(1+(l\sigma(x))^{n})^{\frac{\gamma\beta}{n}}}
{(1+\varpi_{n}(l\sigma(x))^{n})^{\gamma}\varpi_{n}(l\sigma(x))^{n}}D_{i}\left(\int_{B(x,l\sigma(x))}f(y)\mathrm{d}y\right).
\end{split}
\end{eqnarray*}
Since by the chain rule, it yields that
$$
D_{i}\left(\int_{B(x,l\sigma(x))}f(y)\mathrm{d}y\right)
=lD_{i}\sigma(x)\int_{\partial{B(x,l\sigma(x))}}f(y)\mathrm{d}\mathcal{H}^{n-1}(y)+\int_{B(x,l\sigma(x))}D_{i}f(y)\mathrm{d}y,
$$
here we have used the fact that $\frac{\partial}{\partial(l\sigma(x))}\int_{B(x,l\sigma(x))}f(y)\mathrm{d}y
=\int_{\partial{B(x,l\sigma(x))}}f(y)\mathrm{d}\mathcal{H}^{n-1}(y)$.
Then
\begin{align}\label{301}
\nabla{K_{\varphi,\beta}^{l}f(x)}
&=n\left(\frac{\gamma\beta}{n}-\gamma\right)(1+(l\sigma(x))^{n})^{\frac{\gamma\beta}{n}-\gamma-1}l^{n}\sigma(x)^{n-1}\nabla{\sigma(x)}
(l\sigma(x))^{\beta-n}\int_{B(x,l\sigma(x))}f(y)\mathrm{d}y\notag\\
&~~~+(\beta-n)(1+(l\sigma(x))^{n})^{\frac{\gamma\beta}{n}-\gamma}l^{\beta-n}\sigma(x)^{\beta-n-1}\nabla{\sigma(x)}
\int_{B(x,l\sigma(x))}f(y)\mathrm{d}y\notag\\
&~~~+\frac{(l\sigma(x))^{\beta}(1+(l\sigma(x))^{n})^{\frac{\gamma\beta}{n}}}
{(1+\varpi_{n}(l\sigma(x))^{n})^{\gamma}\varpi_{n}(l\sigma(x))^{n}}
l\nabla\sigma(x)\int_{\partial{B(x,l\sigma(x))}}f(y)\mathrm{d}\mathcal{H}^{n-1}(y)\notag\\
&~~~+\frac{(l\sigma(x))^{\beta}(1+(l\sigma(x))^{n})^{\frac{\gamma\beta}{n}}}
{(1+\varpi_{n}(l\sigma(x))^{n})^{\gamma}\varpi_{n}(l\sigma(x))^{n}}\int_{B(x,l\sigma(x))}\nabla{f(y)}\mathrm{d}y.
\end{align}
Thus,
\begin{align}\label{3001}
&~~~~\nabla{K_{\varphi,\beta}^{l}f(x)}\notag\\
&=\frac{\gamma\beta-\gamma{n}}{1+(l\sigma(x))^{n}}\frac{\nabla{\sigma(x)}}{\sigma(x)}
\frac{(l\sigma(x))^{\beta}(1+(l\sigma(x))^{n})^{\frac{\gamma\beta}{n}}}
{\varphi(|B(x,l\sigma(x))|)}\int_{B(x,l\sigma(x))}f(y)\mathrm{d}y\notag\\
&~~~+\beta\frac{\nabla{\sigma(x)}}{\sigma(x)}
\frac{(l\sigma(x))^{\beta}(1+(l\sigma(x))^{n})^{\frac{\gamma\beta}{n}}}
{\varphi(|B(x,l\sigma(x))|)|B(x,l\sigma(x))|}\int_{B(x,l\sigma(x))}f(y)\mathrm{d}y\notag\\
&~~~+\frac{(l\sigma(x))^{\beta}(1+(l\sigma(x))^{n})^{\frac{\gamma\beta}{n}}}
{\varphi(|B(x,l\sigma(x))|)|B(x,l\sigma(x))|}\int_{B(x,l\sigma(x))}\nabla{f(y)}\mathrm{d}y\notag\\
&~~~+n\frac{\nabla{\sigma(x)}}{\sigma(x)}
\frac{(l\sigma(x))^{\beta}(1+(l\sigma(x))^{n})^{\frac{\gamma\beta}{n}}}
{\varphi(|B(x,l\sigma(x))|)}\notag\\
&~~~\times\left(\frac{1}{|\partial{B(x,l\sigma(x))|}}
\int_{\partial{B(x,l\sigma(x))}}f(y)\mathrm{d}\mathcal{H}^{n-1}(y)-
\frac{1}{|{B(x,l\sigma(x))|}}\int_{B(x,l\sigma(x))}{f(y)}\mathrm{d}y\right).
\end{align}
Applying Green's first identity there holds
\begin{align}\label{302}
\int_{\partial{B(x,l\sigma(x))}}{f(y)}\frac{\partial{\mu}}{\partial{\nu}}(y)\mathrm{d}\mathcal{H}^{n-1}(y)
=\int_{{B(x,l\sigma(x))}}(f(y)\Delta{\mu}(y)+\nabla{f}(y)\cdot{\nabla{\mu}(y)})\mathrm{d}y,
\end{align}
where $\nu(y)=\frac{y-x}{l\sigma(x)}$ denotes the unit outer normal of $B(x,l\sigma(x))$. If we choose $\mu(y)=\frac{|y-x|^{2}}{2}$, then
$\frac{\partial{\mu}}{\partial{\nu}}(y)=l\sigma(x)$, $\Delta{\mu}(y)=n$, $\nabla{\mu}(y)=y-x$.
It follows from the inequality \eqref{302} that
\begin{align}\label{3002}
&~~~\frac{1}{|\partial{B(x,l\sigma(x))|}}
\int_{\partial{B(x,l\sigma(x))}}f(y)\mathrm{d}\mathcal{H}^{n-1}(y)-
\frac{1}{|{B(x,l\sigma(x))|}}\int_{B(x,l\sigma(x))}{f(y)}\mathrm{d}y\notag\\
&=\frac{1}{n}\frac{1}{|{B(x,l\sigma(x))|}}\int_{{B(x,l\sigma(x))}}(y-x)\cdot\nabla{f}(y)\mathrm{d}y.
\end{align}

If $0<l\sigma(x)<1$ and $0\leq\gamma\leq{n}$, then we can easily check that $\frac{(l\sigma(x))^{n}}{(1+(l\sigma(x))^{n})^{1-\frac{\gamma}{n}}}\leq{1}$. Thus, from $0<l<1$ and $|\nabla\sigma(x)|=1$ for almost every $x\in{\Omega}$, we have
\begin{align}\label{303}
|\nabla{K_{\varphi,\beta}^{l}f(x)}|
&\leq\gamma{n}\frac{(l\sigma(x))^{n}}{(1+(l\sigma(x))^{n})^{1-\frac{\gamma}{n}}}
\frac{(l\sigma(x))^{\beta-1}(1+(l\sigma(x))^{n})^{\frac{\gamma}{n}(\beta-1)}}
{\varphi(|B(x,l\sigma(x))|)|B(x,l\sigma(x))|}\int_{B(x,l\sigma(x))}|f(y)|\mathrm{d}y\notag\\
&~~~+\beta\varphi(|B(x,l\sigma(x))|)^{\frac{1}{n}}
\frac{(l\sigma(x))^{\beta-1}(1+(l\sigma(x))^{n})^{\frac{\gamma}{n}(\beta-1)}}
{\varphi(|B(x,l\sigma(x))|)|B(x,l\sigma(x))|}\int_{B(x,l\sigma(x))}|f(y)|\mathrm{d}y\notag\\
&~~~+\frac{(l\sigma(x))^{\beta}(1+(l\sigma(x))^{n})^{\frac{\gamma\beta}{n}}}{\varphi(|B(x,l\sigma(x))|)|B(x,l\sigma(x))|}
\int_{B(x,l\sigma(x))}|\nabla{f(y)}|\mathrm{d}y\notag\\
&~~~+\frac{1}{\sigma(x)}\frac{(l\sigma(x))^{\beta}(1+(l\sigma(x))^{n})^{\frac{\gamma\beta}{n}}}{\varphi(|B(x,l\sigma(x))|)|B(x,l\sigma(x))|}
\int_{{B(x,l\sigma(x))}}|y-x||\nabla{f}(y)|\mathrm{d}y\notag\\
&\leq{(\gamma{n}+2\beta){M}_{\varphi,\beta-1,\Omega}{f}(x)+2{M}_{\varphi,\beta,\Omega}{|\nabla{f}|}(x)}.
\end{align}

If $l\sigma(x)\geq1$, then using the inequalities \eqref{3001} and \eqref{3002}, we obtain
\begin{eqnarray*}
\begin{split}
|\nabla{K_{\varphi,\beta}^{l}f(x)}|
&\leq \gamma{n}\frac{(l\sigma(x))^{n}}{1+(l\sigma(x))^{n}}\frac{1}{l\sigma(x)}
\frac{(l\sigma(x))^{\beta}(1+(l\sigma(x))^{n})^{\frac{\gamma\beta}{n}}}
{\varphi(|B(x,l\sigma(x))|)|B(x,l\sigma(x))|}\int_{B(x,l\sigma(x))}|f(y)|\mathrm{d}y\notag\\
&~~~+\beta\frac{1}{l\sigma(x)}
\frac{(l\sigma(x))^{\beta}(1+(l\sigma(x))^{n})^{\frac{\gamma\beta}{n}}}
{\varphi(|B(x,l\sigma(x))|)|B(x,l\sigma(x))|}\int_{B(x,l\sigma(x))}|f(y)|\mathrm{d}y\notag\\
&~~~+\frac{(l\sigma(x))^{\beta}(1+(l\sigma(x))^{n})^{\frac{\gamma\beta}{n}}}{\varphi(|B(x,l\sigma(x))|)|B(x,l\sigma(x))|}
\int_{B(x,l\sigma(x))}|\nabla{f(y)}|\mathrm{d}y\notag\\
&~~~+\frac{1}{\sigma(x)}\frac{(l\sigma(x))^{\beta}(1+(l\sigma(x))^{n})^{\frac{\gamma\beta}{n}}}{\varphi(|B(x,l\sigma(x))|)|B(x,l\sigma(x))|}
\int_{{B(x,l\sigma(x))}}|y-x||\nabla{f}(y)|\mathrm{d}y\notag\\
&\leq{(\gamma{n}+\beta){M}_{\varphi,\beta,\Omega}{f}(x)+2{M}_{\varphi,\beta,\Omega}{|\nabla{f}|}(x)}.
\end{split}
\end{eqnarray*}
This derives the conclusions of $(i)$ and $(ii)$ for smooth function.

Now, we consider the case $f\in{W^{1,p}(\Omega)}$. We only give the proof of $(i)$ because the conclusion of $(ii)$ can be easily obtained by using the similar arguments as the case $0<l\sigma(x)<1$. For $f\in{W^{1,p}(\Omega)}$, there exists a sequence $\{\psi_{j}\}_{j=1}^{\infty}\in{W^{1,p}(\Omega)\cap{\mathcal{C}^{\infty}(\Omega)}}$ such that $\psi_{j}\rightarrow{f}$ on $W^{1,p}(\Omega)$ as $j\rightarrow{\infty}$, then for each $x\in{\Omega}$, one can get that $K_{\varphi,\beta}^{l}f(x)=\lim\limits_{j\rightarrow{\infty}}K_{\varphi,\beta}^{l}\psi_{j}(x)$. Thus, applying the inequality \eqref{303} we have
\begin{align}\label{304}
|\nabla{K_{\varphi,\beta}^{l}\psi_{j}}(x)|
\leq{(\gamma{n}+2\beta){M}_{\varphi,\beta-1,\Omega}{\psi_{j}}(x)+2{M}_{\varphi,\beta,\Omega}{|\nabla\psi_{j}|}(x)}.
\end{align}

If the condition $(a)$ holds. Let $\frac{1}{\tilde{q}}=\frac{1}{p}-\frac{1}{n}$. Clearly, $\frac{1}{q}=\frac{1}{\tilde{q}}-\frac{\beta-1}{n}$. Then by the inequality \eqref{304}, Lemma \ref{lem2.1}, and the Sobolev embedding property ($\|{\psi_{j}}\|_{L^{\tilde{{q}}}(\Omega)}\leq{C\|{\psi_{j}}\|_{W^{1,p}(\Omega)}}$), we have
\begin{align}\label{3004}
\|\nabla{K_{\varphi,\beta}^{l}\psi_{j}}\|_{L^{q}(\Omega)}
&\leq{(\gamma{n}+2\beta)\|{M}_{\varphi,\beta-1,\Omega}{\psi_{j}}\|_{L^{q}(\Omega)}
+2\|{M}_{\varphi,\beta,\Omega}{|\nabla\psi_{j}|}\|_{L^{q}(\Omega)}}\notag\\
&\leq{C\|{\psi_{j}}\|_{L^{\tilde{{q}}}(\Omega)}
+C\|{\nabla\psi_{j}}\|_{L^{p}(\Omega)}}\notag\\
&\leq{C\|{\psi_{j}}\|_{W^{1,p}(\Omega)}}.
\end{align}
Thus, $\{|\nabla{K_{\varphi,\beta}^{l}\psi_{j}}|\}_{j=1}^{\infty}\in{L^{q}(\Omega)}$ is a bounded sequence and has a weakly converging subsequence $\{|\nabla{K_{\varphi,\beta}^{l}\psi_{j_{k}}}|\}_{k=1}^{\infty}\in{L^{q}(\Omega)}$. Since $K_{\varphi,\beta}^{l}\psi_{j}$ converges pointwise to $K_{\varphi,\beta}^{l}f$, then we know that $|\nabla{K_{\varphi,\beta}^{l}}f|$ exists and $|\nabla{K_{\varphi,\beta}^{l}\psi_{j_{k}}}|$ converges weakly to $|\nabla{K_{\varphi,\beta}^{l}f}|$ on $L^{q}(\Omega)$ as $k\rightarrow{\infty}$. Lemma \ref{lem2.1} together with the Sobolev embedding property and the sublinearity of maximal operators implies that
\begin{eqnarray*}
\begin{split}
\|{M}_{\varphi,\beta-1,\Omega}{\psi_{{j}_{k}}}-{M}_{\varphi,\beta-1,\Omega}f\|_{L^{q}(\Omega)}
&\leq{\|{M}_{\varphi,\beta-1,\Omega}({\psi_{{j}_{k}}}-f)\|_{L^{q}(\Omega)}}\\
&\leq{C\|{\psi_{{j}_{k}}}-f\|_{L^{\tilde{{q}}}(\Omega)}}\\
&\leq{C\|{\psi_{{j}_{k}}}-f\|_{W^{1,p}(\Omega)}},
\end{split}
\end{eqnarray*}
and
\begin{eqnarray*}
\begin{split}
\|{M}_{\varphi,\beta,\Omega}|\nabla\psi_{{j}_{k}}|-{M}_{\varphi,\beta,\Omega}|\nabla{f}|\|_{L^{q}(\Omega)}
\leq{C\||\nabla{\psi_{{j}_{k}}}|-|\nabla{f}|\|_{L^{p}(\Omega)}}
\leq{C\|{\psi_{{j}_{k}}}-f\|_{W^{1,p}(\Omega)}}.
\end{split}
\end{eqnarray*}

If the condition $(b)$ holds. Assume that $\frac{1}{q^{\ast}}=\frac{1}{p}-\frac{\beta}{n}$. Obviously, $q<{q^{\ast}}$. Then from H\"{o}lder's inequality, $|\Omega|<{\infty}$, the inequality \eqref{304}, and Lemma \ref{lem2.1}, it follows that
\begin{align}\label{3005}
\|\nabla{K_{\varphi,\beta}^{l}\psi_{j}}\|_{L^{q}(\Omega)}
&\leq{(\gamma{n}+2\beta)\|{M}_{\varphi,\beta-1,\Omega}{\psi_{j}}\|_{L^{q}(\Omega)}
+2\|{M}_{\varphi,\beta,\Omega}{|\nabla\psi_{j}|}\|_{L^{q}(\Omega)}}\notag\\
&\leq{C\|{\psi_{j}}\|_{L^{p}(\Omega)}
+C|\Omega|^{\frac{1}{q}-\frac{1}{q^{\ast}}}\|{M}_{\varphi,\beta,\Omega}
{|\nabla\psi_{j}|}\|_{L^{q^{\ast}}(\Omega)}}\notag\\
&\leq{C\|{\psi_{j}}\|_{W^{1,p}(\Omega)}}.
\end{align}
Thus, $\{|\nabla{K_{\varphi,\beta}^{l}\psi_{j}}|\}_{j=1}^{\infty}\in{L^{q}(\Omega)}$ is a bounded sequence and has a weakly converging subsequence $\{|\nabla{K_{\varphi,\beta}^{l}\psi_{j_{k}}}|\}_{k=1}^{\infty}\in{L^{q}(\Omega)}$. Since $K_{\varphi,\beta}^{l}\psi_{j}\xrightarrow{j\rightarrow{\infty}}{K_{\varphi,\beta}^{l}f}$, then $|\nabla{K_{\varphi,\beta}^{l}}f|$ exists and as $k\rightarrow{\infty}$, $|\nabla{K_{\varphi,\beta}^{l}\psi_{j_{k}}}|\rightarrow{|\nabla{K_{\varphi,\beta}^{l}f}|}$ weakly on $L^{q}(\Omega)$. Applying Lemma \ref{lem2.1}, $|\Omega|<{\infty}$, H\"{o}lder's inequality, and the sublinearity of maximal operators, we have
\begin{eqnarray*}
\begin{split}
\|{M}_{\varphi,\beta-1,\Omega}{\psi_{{j}_{k}}}-{M}_{\varphi,\beta-1,\Omega}f\|_{L^{q}(\Omega)}
&\leq{\|{M}_{\varphi,\beta-1,\Omega}({\psi_{{j}_{k}}}-f)\|_{L^{q}(\Omega)}}\\
&\leq{C\|{\psi_{{j}_{k}}}-f\|_{L^{p}(\Omega)}}\\
&\leq{C\|{\psi_{{j}_{k}}}-f\|_{W^{1,p}(\Omega)}},
\end{split}
\end{eqnarray*}
and
\begin{eqnarray*}
\begin{split}
\|{M}_{\varphi,\beta,\Omega}|\nabla\psi_{{j}_{k}}|-{M}_{\varphi,\beta,\Omega}|\nabla{f}|\|_{L^{q}(\Omega)}
\leq{C\|{M}_{\varphi,\beta,\Omega}(|\nabla\psi_{{j}_{k}}|-|\nabla{f}|)\|_{L^{q^{\ast}}(\Omega)}}
\leq{C\|{\psi_{{j}_{k}}}-f\|_{W^{1,p}(\Omega)}}.
\end{split}
\end{eqnarray*}
Therefore, $(\gamma{n}+2\beta){M}_{\varphi,\beta-1,\Omega}{\psi_{{j}_{k}}}
+2{M}_{\varphi,\beta,\Omega}|\nabla\psi_{{j}_{k}}|\rightarrow{(\gamma{n}+2\beta)
{M}_{\varphi,\beta-1,\Omega}f+2{M}_{\varphi,\beta,\Omega}|\nabla{f}|}$ on $L^{q}(\Omega)$ as $k\rightarrow{\infty}$. Then applying Lemma \ref{lem2.3} to the inequality \eqref{304} with $g_{k}=|\nabla{K_{\varphi,\beta}^{l}\psi_{j_{k}}}|$ and $f_{k}=(\gamma{n}+2\beta){M}_{\varphi,\beta-1,\Omega}{\psi_{j_{k}}}
+2{M}_{\varphi,\beta,\Omega}{|\nabla\psi_{j_{k}}|}$, we get the inequality \eqref{303}. The proof is complete.
\end{proof}

\begin{lemma}\label{lem4.1}
Let $f\in{L^{p}(\Omega)}$, $n\geq{2}$, $p>{\frac{n}{n-1}}$, and $0<l<1$. Then

(i)~~If $0<l\sigma(x)<1$, $1\leq{\beta}<\min\{{n-\frac{2n}{p(n-1)}},\frac{n-1}{p}\}+1$, and $\frac{1}{q}=\frac{1}{p}-\frac{\beta-1}{n}$, then $|\nabla{K_{\varphi,\beta}^{l}f}|\in{L^{q}(\Omega)}$. Moreover, for almost every $x\in{\Omega}$, there is a positive constant $C$ which depends on $\gamma$ and $n$, such that
$$
|\nabla{K_{\varphi,\beta}^{l}f}(x)|
\leq{C({M}_{\varphi,\beta-1,\Omega}{f}(x)+{S}_{\varphi,\beta-1,\Omega}{f}(x))}.
$$

(ii)~~If $l\sigma(x)\geq1$, $1\leq{\beta}<\min\{{n-\frac{2n}{p(n-1)}},\frac{n-1}{p}\}$, and $\frac{1}{q}=\frac{1}{p}-\frac{\beta}{n}$, then $|\nabla{K_{\varphi,\beta}^{l}f}|\in{L^{q}(\Omega)}$. Moreover, there is a constant $C>0$, depending on $\gamma$ and $n$, such that for almost every $x\in{\Omega}$, there holds
$$
|\nabla{K_{\varphi,\beta}^{l}f}(x)|
\leq{C({M}_{\varphi,\beta,\Omega}{f}(x)+{S}_{\varphi,\beta,\Omega}{f}(x))}.
$$
\end{lemma}

\begin{proof}
Assume that $f\in{L^{p}(\Omega)\cap{\mathcal{C}^{\infty}(\Omega)}}$. Using Gauss's theorem, there holds
$$
\int_{B(x,l\sigma(x))}\nabla{f(y)}\mathrm{d}y=\int_{\partial{B(x,l\sigma(x))}}f(y)\nu(y)\mathrm{d}\mathcal{H}^{n-1}(y),
$$
where $\nu(y)=\frac{y-x}{l\sigma(x)}$ denotes the unit outer normal of $B(x,l\sigma(x))$. This equality together with the inequality \eqref{301} implies that
\begin{eqnarray*}
\begin{split}
\nabla{K_{\varphi,\beta}^{l}f(x)}
&=\frac{\gamma\beta-\gamma{n}}{1+(l\sigma(x))^{n}}\frac{\nabla{\sigma(x)}}{\sigma(x)}
\frac{(l\sigma(x))^{\beta}(1+(l\sigma(x))^{n})^{\frac{\gamma\beta}{n}}}
{\varphi(|B(x,l\sigma(x))|)}\int_{B(x,l\sigma(x))}f(y)\mathrm{d}y\\
&~~~+(\beta-n)\frac{\nabla{\sigma(x)}}{\sigma(x)}
\frac{(l\sigma(x))^{\beta}(1+(l\sigma(x))^{n})^{\frac{\gamma\beta}{n}}}
{\varphi(|B(x,l\sigma(x))|)|B(x,l\sigma(x))|}\int_{B(x,l\sigma(x))}f(y)\mathrm{d}y\\
&~~~+\frac{(l\sigma(x))^{\beta}(1+(l\sigma(x))^{n})^{\frac{\gamma\beta}{n}}}
{\varphi(|B(x,l\sigma(x))|)|B(x,l\sigma(x))|}
l\nabla\sigma(x)\int_{\partial{B(x,l\sigma(x))}}f(y)\mathrm{d}\mathcal{H}^{n-1}(y)\\
&~~~+\frac{(l\sigma(x))^{\beta}(1+(l\sigma(x))^{n})^{\frac{\gamma\beta}{n}}}
{\varphi(|B(x,l\sigma(x))|)|B(x,l\sigma(x))|}
\int_{\partial{B(x,l\sigma(x))}}f(y)\nu(y)\mathrm{d}\mathcal{H}^{n-1}(y).
\end{split}
\end{eqnarray*}
From $0<l<1$ and $|\nabla\sigma(x)|=1$ for almost every $x\in{\Omega}$, we have
\begin{eqnarray*}
\begin{split}
|\nabla{K_{\varphi,\beta}^{l}f(x)}|
&\leq (\gamma{n}+n)\frac{1}{l\sigma(x)}
\frac{(l\sigma(x))^{\beta}(1+(l\sigma(x))^{n})^{\frac{\gamma\beta}{n}}}{\varphi(|B(x,l\sigma(x))|)|B(x,l\sigma(x))|}\int_{B(x,l\sigma(x))}|f(y)|\mathrm{d}y\\
&~~~+\frac{\varphi(|\partial{B(x,l\sigma(x))}|)|\partial{B(x,l\sigma(x))}|}{\varphi(|B(x,l\sigma(x))|)|B(x,l\sigma(x))|}
\frac{(l\sigma(x))^{\beta}(1+(l\sigma(x))^{n})^{\frac{\gamma\beta}{n}}}{\varphi(|\partial{B(x,l\sigma(x))}|)|\partial{B(x,l\sigma(x))}|}\\
&~~~~~~\times\int_{\partial{B(x,l\sigma(x))}}|f(y)|\mathrm{d}\mathcal{H}^{n-1}(y)\\
&~~~+\frac{\varphi(|\partial{B(x,l\sigma(x))}|)|\partial{B(x,l\sigma(x))}|}{\varphi(|B(x,l\sigma(x))|)|B(x,l\sigma(x))|}
\frac{(l\sigma(x))^{\beta}(1+(l\sigma(x))^{n})^{\frac{\gamma\beta}{n}}}{\varphi(|\partial{B(x,l\sigma(x))}|)|\partial{B(x,l\sigma(x))}|}\\
&~~~~~~\times\int_{\partial{B(x,l\sigma(x))}}|f(y)||\nu(y)|\mathrm{d}\mathcal{H}^{n-1}(y).
\end{split}
\end{eqnarray*}

If $0<l\sigma(x)<1$, then $\varphi(|B(x,l\sigma(x))|)\sim{1}$. Thus
$$
\frac{\varphi(|\partial{B(x,l\sigma(x))}|)|\partial{B(x,l\sigma(x))}|}
{\varphi(|B(x,l\sigma(x))|)|B(x,l\sigma(x))|}
\lesssim(1+n\varpi_{n}(l\sigma(x))^{n-1})^{\gamma}\frac{n\varpi_{n}(l\sigma(x))^{n-1}}{\varpi_{n}(l\sigma(x))^{n}}
\leq{\frac{n(1+n)^{\gamma}}{l\sigma(x)}}.
$$
Hence,
\begin{align}\label{401}
&~~~~|\nabla{K_{\varphi,\beta}^{l}f(x)}|\notag\\
&\leq C\varphi(|B(x,l\sigma(x))|)^{\frac{1}{n}}
\frac{(l\sigma(x))^{\beta-1}(1+(l\sigma(x))^{n})^{\frac{\gamma}{n}(\beta-1)}}
{\varphi(|B(x,l\sigma(x))|)|B(x,l\sigma(x))|}\int_{B(x,l\sigma(x))}|f(y)|\mathrm{d}y\notag\\
&~~~+C\varphi(|B(x,l\sigma(x))|)^{\frac{1}{n}}
\frac{(l\sigma(x))^{\beta-1}(1+(l\sigma(x))^{n})^{\frac{\gamma}{n}(\beta-1)}}{\varphi(|\partial{B(x,l\sigma(x))}|)|\partial{B(x,l\sigma(x))}|}
\int_{\partial{B(x,l\sigma(x))}}|f(y)|\mathrm{d}\mathcal{H}^{n-1}(y)\notag\\
&\leq{C({M}_{\varphi,\beta-1,\Omega}{f}(x)+{S}_{\varphi,\beta-1,\Omega}{f}(x))}.
\end{align}

If $l\sigma(x)\geq1$, then
$$
\frac{\varphi(|\partial{B(x,l\sigma(x))}|)|\partial{B(x,l\sigma(x))}|}{\varphi(|B(x,l\sigma(x))|)|B(x,l\sigma(x))|}
\leq\left(\frac{1+n\varpi_{n}(l\sigma(x))^{n-1}}{\varpi_{n}(l\sigma(x))^{n}}\right)^{\gamma}\frac{n\varpi_{n}(l\sigma(x))^{n-1}}{\varpi_{n}(l\sigma(x))^{n}}
\leq{n(1+n)^{\gamma}}.
$$
Thus,
\begin{eqnarray*}
\begin{split}
|\nabla{K_{\varphi,\beta}^{l}f(x)}|
&\leq C\frac{(l\sigma(x))^{\beta}(1+(l\sigma(x))^{n})^{\frac{\gamma\beta}{n}}}
{\varphi(|B(x,l\sigma(x))|)|B(x,l\sigma(x))|}\int_{B(x,l\sigma(x))}|f(y)|\mathrm{d}y\\
&~~~+C\frac{(l\sigma(x))^{\beta}(1+(l\sigma(x))^{n})^{\frac{\gamma\beta}{n}}}{\varphi(|\partial{B(x,l\sigma(x))}|)|\partial{B(x,l\sigma(x))}|}
\int_{\partial{B(x,l\sigma(x))}}|f(y)|\mathrm{d}\mathcal{H}^{n-1}(y)\\
&\leq{C({M}_{\varphi,\beta,\Omega}{f}(x)+{S}_{\varphi,\beta,\Omega}{f}(x))}.
\end{split}
\end{eqnarray*}
So we obtain the conclusions of $(i)$ and $(ii)$ as $f\in{L^{p}(\Omega)\cap{\mathcal{C}^{\infty}(\Omega)}}$.

We now consider the case $f\in{L^{p}(\Omega)}$. Since the proof of $(ii)$ is similar to the proof of $(i)$, so it suffices to prove the case $0<l\sigma(x)<1$. For $f\in{L^{p}(\Omega)}$, there exists a sequence $\{\psi_{j}\}_{j=1}^{\infty}$ of functions on ${L^{p}(\Omega)\cap{\mathcal{C}^{\infty}(\Omega)}}$ such that $\psi_{j}\rightarrow{f}$ on $L^{p}(\Omega)$ as $j\rightarrow{\infty}$, then for every $x\in{\Omega}$, one can conclude that $K_{\varphi,\beta}^{l}f(x)=\lim\limits_{j\rightarrow{\infty}}K_{\varphi,\beta}^{l}\psi_{j}(x)$. Thus, according to the inequality \eqref{401} we have
\begin{align}\label{402}
|\nabla{K_{\varphi,\beta}^{l}\psi_{j}}(x)|
\leq{C({M}_{\varphi,\beta-1,\Omega}{\psi_{j}}(x)+{S}_{\varphi,\beta-1,\Omega}{\psi_{j}}(x))},
\end{align}
for almost every $x\in{\Omega}$ and $j=1,2,\cdots$, then from this inequality, and Lemmas \ref{lem2.1} and \ref{lem2.2}, we have
\begin{align}\label{403}
\|\nabla{K_{\varphi,\beta}^{l}\psi_{j}}\|_{L^{q}(\Omega)}
\leq{C(\|{M}_{\varphi,\beta-1,\Omega}{\psi_{j}}\|_{L^{q}(\Omega)}+\|{S}_{\varphi,\beta-1,\Omega}{\psi_{j}}\|_{L^{q}(\Omega)})}
\leq{C\|\psi_{j}\|_{L^{p}(\Omega)}}.
\end{align}
Thus, $\{|\nabla{K_{\varphi,\beta}^{l}\psi_{j}}|\}_{j=1}^{\infty}$ is a bounded sequence on ${L^{q}(\Omega)}$ and has a weakly converging subsequence $\{|\nabla{K_{\varphi,\beta}^{l}\psi_{j_{k}}}|\}_{k=1}^{\infty}\in{L^{q}(\Omega)}$. Since $K_{\varphi,\beta}^{l}\psi_{j}\xrightarrow{j\rightarrow{\infty}}{K_{\varphi,\beta}^{l}f}$, then $|\nabla{K_{\varphi,\beta}^{l}}f|$ exists and $|\nabla{K_{\varphi,\beta}^{l}\psi_{j_{k}}}|\xrightarrow{k\rightarrow{\infty}}{|\nabla{K_{\varphi,\beta}^{l}f}|}$ on $L^{q}(\Omega)$. Using the sublinearity of maximal operators and Lemmas \ref{lem2.1}-\ref{lem2.2}, we get that
\begin{eqnarray*}
\begin{split}
\|{M}_{\varphi,\beta-1,\Omega}{\psi_{{j}_{k}}}-{M}_{\varphi,\beta-1,\Omega}f\|_{L^{q}(\Omega)}
\leq{\|{M}_{\varphi,\beta-1,\Omega}({\psi_{{j}_{k}}}-f)\|_{L^{q}(\Omega)}}
\leq{C\|{\psi_{{j}_{k}}}-f\|_{L^{p}(\Omega)}},
\end{split}
\end{eqnarray*}
and
\begin{eqnarray*}
\begin{split}
\|{S}_{\varphi,\beta-1,\Omega}{\psi_{{j}_{k}}}-{S}_{\varphi,\beta-1,\Omega}f\|_{L^{q}(\Omega)}
\leq{\|{S}_{\varphi,\beta-1,\Omega}({\psi_{{j}_{k}}}-f)\|_{L^{q}(\Omega)}}
\leq{C\|{\psi_{{j}_{k}}}-f\|_{L^{p}(\Omega)}}.
\end{split}
\end{eqnarray*}
Therefore, ${M}_{\varphi,\beta-1,\Omega}{\psi_{{j}_{k}}}+{S}_{\varphi,\beta-1,\Omega}{\psi_{{j}_{k}}}\rightarrow{{M}_{\varphi,\beta-1,\Omega}f
+{S}_{\varphi,\beta-1,\Omega}f}$ on $L^{q}(\Omega)$ as $k\rightarrow{\infty}$. Then applying Lemma \ref{lem2.3} to the inequality \eqref{402} with $g_{k}=|\nabla{K_{\varphi,\beta}^{l}\psi_{j_{k}}}|$ and $f_{k}=C({M}_{\varphi,\beta-1,\Omega}{\psi_{j_{k}}}+{S}_{\varphi,\beta-1,\Omega}{\psi_{j_{k}}})$, we get the inequality \eqref{401}. The proof is finished.
\end{proof}

Now, we give the proofs of Theorem \ref{thm0} and Theorem \ref{thm2}.

\begin{proof}[The proof of Theorem \ref{thm0}]
We only give the proof of $(i)$ because $(ii)$ can be proved similarly. Without loss of generality, we may assume that ${f}\geq{0}$. Let $l_{k}(k=1,2,\cdots)$ be an enumeration of the rationals between 0 and 1. Denote ${M}_{\varphi,\beta,\Omega}{f}(x)=\sup\limits_{k\geq1}K_{\varphi,\beta}^{l_{k}}f(x)$ for each $x\in{\Omega}$. It follows from Lemma \ref{lem3.1} $(i)$ that $|\nabla{K_{\varphi,\beta}^{l_{k}}f}|\in{L^{q}(\Omega)}$ for all $k\geq1$ and for almost every $x\in{\Omega}$,
$$
|\nabla{K_{\varphi,\beta}^{l_{k}}f}(x)|
\leq{(\gamma{n}+2\beta){M}_{\varphi,\beta-1,\Omega}{f}(x)+2{M}_{\varphi,\beta,\Omega}{|\nabla{f}|}(x)}.
$$
For $k\geq1$, we define the function $u_{k}:\Omega\rightarrow{[-\infty,\infty]}$ by $u_{k}(x)=\max\limits_{1\leq{i}\leq{k}}K_{\varphi,\beta}^{l_{i}}f(x)$. Then $\{u_{k}\}_{k=1}^{\infty}$ is an increasing sequence of functions converging to ${M}_{\varphi,\beta,\Omega}{f}$ pointwisely. Thus, for almost every $x\in{\Omega}$ and $k=1,2,\cdots$, we obtain
\begin{align}\label{306}
|\nabla{u_{k}}(x)|
\leq\max\limits_{1\leq{i}\leq{k}}|\nabla{K_{\varphi,\beta}^{l_{i}}f}(x)|
\leq{(\gamma{n}+2\beta){M}_{\varphi,\beta-1,\Omega}{f}(x)+2{M}_{\varphi,\beta,\Omega}{|\nabla{f}|}(x)}.
\end{align}
From the fact that ${u_{k}}(x)\leq{{M}_{\varphi,\beta,\Omega}{f}(x)}$ for $k=1,2,\cdots$, the inequalities \eqref{3004} (resp., \eqref{3005}) and \eqref{306}, and Lemma \ref{lem2.1}, it follows that
\begin{eqnarray*}
\begin{split}
\|{u_{k}}\|_{W^{1,q}(\Omega)}
=\|{u_{k}}\|_{L^{q}(\Omega)}+\|\nabla{u_{k}}\|_{L^{q}(\Omega)}
\leq{C\|f\|_{L^{p}(\Omega)}+C\|f\|_{W^{1,p}(\Omega)}}
\leq{C\|f\|_{W^{1,p}(\Omega)}}.
\end{split}
\end{eqnarray*}
Thus, $\{{u_{k}}\}_{k=1}^{\infty}$ is a bounded sequence on $W^{1,q}(\Omega)$ such that $u_{k}\rightarrow{{M}_{\varphi,\beta,\Omega}{f}}$ as $k\rightarrow{\infty}$. It follows from the weak compactness argument that ${M}_{\varphi,\beta,\Omega}{f}\in{W^{1,q}(\Omega)}$ and $|\nabla{u_{k}}|\rightarrow{|\nabla{{M}_{\varphi,\beta,\Omega}{f}}|}$ on $L^{q}(\Omega)$ as $k\rightarrow{\infty}$. The rest of the proof applies the same arguments as in the final part of the proof of Lemma \ref{lem3.1} for the inequality \eqref{306}. This completes the proof of Theorem \ref{thm0}.
\end{proof}

\begin{proof}[The proof of Theorem \ref{thm2}]
The proof is analogous to the proof of Theorem \ref{thm0}, but with Lemma \ref{lem3.1} and the inequality \eqref{3004} replaced by Lemma \ref{lem4.1} and the inequality \eqref{403}, respectively. Then we can yield the desired results. The details are omitted here.
\end{proof}

\section{Proofs of Theorems \ref{thm3}-\ref{thm5}}

\noindent

\begin{proof}[The proof of Theorem \ref{thm3}]
We consider two cases $0<{r}<1$ and $1\leq{r}<\mathrm{dist}(x,\Omega^{c})$. As $0<{r}<1$, let $\frac{1}{\tilde{q}}=\frac{1}{p}-\frac{1}{n}$. It is obvious that $\frac{1}{q}=\frac{1}{\tilde{q}}-\frac{\beta-1}{n}$. Then by Lemma \ref{lem2.1}, Theorem \ref{thm0} $(i)$, and the Sobolev embedding property, we have
\begin{eqnarray*}
\begin{split}
\|{M}_{\varphi,\beta,\Omega}{f}\|_{W^{1,q}(\Omega)}
&\leq{\|{M}_{\varphi,\beta,\Omega}{f}\|_{L^{q}(\Omega)}
+(\gamma{n}+2\beta)\|{M}_{\varphi,\beta-1,\Omega}{f}\|_{L^{q}(\Omega)}
+2\|{M}_{\varphi,\beta,\Omega}{|\nabla{f}|}\|_{L^{q}(\Omega)}}\\
&\leq{C\|f\|_{L^{p}(\Omega)}
+C\|f\|_{L^{\tilde{q}}(\Omega)}
+C\|\nabla{f}\|_{L^{p}(\Omega)}}\\
&\leq{C\|f\|_{W^{1,p}(\Omega)}}.
\end{split}
\end{eqnarray*}

As $1\leq{r}<\mathrm{dist}(x,\Omega^{c})$, then using Lemma \ref{lem2.1} and Theorem \ref{thm0} $(ii)$, we have
\begin{eqnarray*}
\begin{split}
\|{M}_{\varphi,\beta,\Omega}{f}\|_{W^{1,q}(\Omega)}
&\leq{\|{M}_{\varphi,\beta,\Omega}{f}\|_{L^{q}(\Omega)}
+(\gamma{n}+\beta)\|{M}_{\varphi,\beta,\Omega}{f}\|_{L^{q}(\Omega)}
+2\|{M}_{\varphi,\beta,\Omega}{|\nabla{f}|}\|_{L^{q}(\Omega)}}\\
&\leq{C\|f\|_{W^{1,p}(\Omega)}}.
\end{split}
\end{eqnarray*}
This yields the conclusion of Theorem \ref{thm3}.
\end{proof}

\begin{proof}[The proof of Theorem \ref{thm4}]
Similarly as in the proof of Theorem \ref{thm3}, we also consider two cases. As $0<{r}<1$, let $\frac{1}{q^{\ast}}=\frac{1}{p}-\frac{\beta}{n}$. Clearly, $q<{q^{\ast}}$. Then from Theorem \ref{thm0} $(i)$, Lemma \ref{lem2.1}, H\"{o}lder's inequality, and $|\Omega|<{\infty}$, it follows that
\begin{eqnarray*}
\begin{split}
\|{M}_{\varphi,\beta,\Omega}{f}\|_{W^{1,q}(\Omega)}
&\leq{\|{M}_{\varphi,\beta,\Omega}{f}\|_{L^{q}(\Omega)}
+(\gamma{n}+2\beta)\|{M}_{\varphi,\beta-1,\Omega}{f}\|_{L^{q}(\Omega)}
+2\|{M}_{\varphi,\beta,\Omega}{|\nabla{f}|}\|_{L^{q}(\Omega)}}\\
&\leq{C|\Omega|^{\frac{1}{q}-\frac{1}{q^{\ast}}}\|{M}_{\varphi,\beta,\Omega}
{f}\|_{L^{q^{\ast}}(\Omega)}+C\|f\|_{L^{p}(\Omega)}
+C|\Omega|^{\frac{1}{q}-\frac{1}{q^{\ast}}}\|{M}_{\varphi,\beta,\Omega}
{|\nabla{f}|}\|_{L^{q^{\ast}}(\Omega)}}\\
&\leq{C\|f\|_{W^{1,p}(\Omega)}}.
\end{split}
\end{eqnarray*}

As $1\leq{r}<\mathrm{dist}(x,\Omega^{c})$, applying Lemma \ref{lem2.1}, H\"{o}lder's inequality, Theorem \ref{thm0} $(ii)$, and $|\Omega|<{\infty}$, there holds
\begin{eqnarray*}
\begin{split}
\|{M}_{\varphi,\beta,\Omega}{f}\|_{W^{1,q}(\Omega)}
&=\|{M}_{\varphi,\beta,\Omega}{f}\|_{L^{q}(\Omega)}
+\|\nabla{M}_{\varphi,\beta,\Omega}{f}\|_{L^{q}(\Omega)}\\
&\leq{C|\Omega|^{\frac{1}{q}-\frac{1}{q^{\ast}}}\|{M}_{\varphi,\beta,\Omega}
{f}\|_{L^{q^{\ast}}(\Omega)}+C|\Omega|^{\frac{1}{q}-\frac{1}{q^{\ast}}}
\|\nabla{M}_{\varphi,\beta,\Omega}{f}\|_{L^{q^{\ast}}(\Omega)}}\\
&\leq{C\|f\|_{L^{p}(\Omega)}+C\|{M}_{\varphi,\beta,\Omega}{f}\|_{L^{q^{\ast}}(\Omega)}
+C\|{M}_{\varphi,\beta,\Omega}{|\nabla{f}|}\|_{L^{q^{\ast}}(\Omega)}}\\
&\leq{C\|f\|_{W^{1,p}(\Omega)}}.
\end{split}
\end{eqnarray*}
The proof is finished.
\end{proof}

\begin{proof}[The proof of Theorem \ref{thm5}]
As in the proof of Theorem \ref{thm3}, we also consider two cases. As $0<{r}<1$, let $\frac{1}{q^{\ast}}=\frac{1}{p}-\frac{\beta}{n}$. Obviously, $q<{q^{\ast}}$. Then it follows from H\"{o}lder's inequality, $|\Omega|<{\infty}$, Theorem \ref{thm2} $(i)$, and Lemmas \ref{lem2.1}-\ref{lem2.2} that
\begin{eqnarray*}
\begin{split}
\|{M}_{\varphi,\beta,\Omega}{f}\|_{W^{1,q}(\Omega)}
&\leq{\|{M}_{\varphi,\beta,\Omega}{f}\|_{L^{q}(\Omega)}
+C\|{M}_{\varphi,\beta-1,\Omega}{f}\|_{L^{q}(\Omega)}
+C\|{S}_{\varphi,\beta-1,\Omega}{f}\|_{L^{q}(\Omega)}}\\
&\leq{C|\Omega|^{\frac{1}{q}-\frac{1}{q^{\ast}}}\|{M}_{\varphi,\beta,\Omega}
{f}\|_{L^{q^{\ast}}(\Omega)}+C\|f\|_{L^{p}(\Omega)}+C\|f\|_{L^{p}(\Omega)}}\\
&\leq{C\|f\|_{L^{p}(\Omega)}}.
\end{split}
\end{eqnarray*}

As $1\leq{r}<\mathrm{dist}(x,\Omega^{c})$, Theorem \ref{thm2} $(ii)$ together with Lemmas \ref{lem2.1}-\ref{lem2.2}, H\"{o}lder's inequality, and $|\Omega|<{\infty}$, we obtain
\begin{eqnarray*}
\begin{split}
\|{M}_{\varphi,\beta,\Omega}{f}\|_{W^{1,q}(\Omega)}
&=\|{M}_{\varphi,\beta,\Omega}{f}\|_{L^{q}(\Omega)}
+\|\nabla{M}_{\varphi,\beta,\Omega}{f}\|_{L^{q}(\Omega)}\\
&\leq{C|\Omega|^{\frac{1}{q}-\frac{1}{q^{\ast}}}\|{M}_{\varphi,\beta,\Omega}
{f}\|_{L^{q^{\ast}}(\Omega)}+C|\Omega|^{\frac{1}{q}-\frac{1}{q^{\ast}}}
\|\nabla{M}_{\varphi,\beta,\Omega}{f}\|_{L^{q^{\ast}}(\Omega)}}\\
&\leq{C\|f\|_{L^{p}(\Omega)}+C\|{M}_{\varphi,\beta,\Omega}{f}\|_{L^{q^{\ast}}(\Omega)}
+C\|{S}_{\varphi,\beta,\Omega}f\|_{L^{q^{\ast}}(\Omega)}}\\
&\leq{C\|f\|_{L^{p}(\Omega)}}.
\end{split}
\end{eqnarray*}
The proof is complete.
\end{proof}

\section{Proof of Theorem \ref{thm6}}

\noindent

To obtain the conclusions of Theorem \ref{thm6}, we need the following lemma.

\begin{lemma}\label{lem6.1}{\textsuperscript{\cite{KM1}}}
Let $\Omega\subset{\mathbb{R}^{n}}$ be an open set and $\Omega\neq{\mathbb{R}^{n}}$. If $f\in{W^{1,p}(\Omega)}$ and $\int_{\Omega}(\frac{f(x)}{\mathrm{dist}(x,\Omega^{c})})^{p}\mathrm{d}x<{\infty}$, then $f\in{W_{0}^{1,p}(\Omega)}$.
\end{lemma}

\begin{proof}[The proof of Theorem \ref{thm6}]
If the condition $(i)$ holds, then by Theorem \ref{thm4}, we know that ${M}_{\varphi,\beta,\Omega}{f}\in{W^{1,q}(\Omega)}$. Next, we consider two cases $0<{r}<1$ and $1\leq{r}<\mathrm{dist}(x,\Omega^{c})$. As $0<{r}<1$, noticing that $(1+r^{n})^{\frac{\gamma}{n}}<2$, then for each $x\in{\Omega}$, there holds
$$
{M}_{\varphi,\beta,\Omega}{f}(x)
<2\mathrm{dist}(x,\Omega^{c}){{M}_{\varphi,\beta-1,\Omega}{f}(x)}.
$$
This inequality together with Lemma \ref{lem2.1} implies that
$$
\int_{\Omega}\left(\frac{{M}_{\varphi,\beta,\Omega}{f}(x)}
{\mathrm{dist}(x,\Omega^{c})}\right)^{q}\mathrm{d}x
<2^{q}\int_{\Omega}({M}_{\varphi,\beta-1,\Omega}{f}(x))^{q}\mathrm{d}x
\leq{C\|{f}\|_{L^{p}(\Omega)}^{q}}
<{\infty}.
$$

As $1\leq{r}<\mathrm{dist}(x,\Omega^{c})$, let $\frac{1}{q^{\ast}}=\frac{1}{p}-\frac{\beta}{n}$. Clearly, $q<{q^{\ast}}$. Then using Lemma \ref{lem2.1}, H\"{o}lder's inequality, and $|\Omega|<{\infty}$, we have
$$
\int_{\Omega}\left(\frac{{M}_{\varphi,\beta,\Omega}{f}(x)}
{\mathrm{dist}(x,\Omega^{c})}\right)^{q}\mathrm{d}x
\leq{C|\Omega|^{1-\frac{q}{q^{\ast}}}
\left(\int_{\Omega}({M}_{\varphi,\beta,\Omega}{f}(x))^{q^{\ast}}\mathrm{d}x\right)^{\frac{q}{q^{\ast}}}}
\leq{C\|{f}\|_{L^{p}(\Omega)}^{q}}
<{\infty}.
$$
Thus, it follows from Lemma \ref{lem6.1} that ${M}_{\varphi,\beta,\Omega}{f}\in{W_{0}^{1,q}(\Omega)}$.

If the condition $(ii)$ holds, then applying Theorem \ref{thm5} instead of Theorem \ref{thm4} and employing the same arguments as in the proof of $(i)$, we can easily obtain the conclusion. The proof is finished.
\end{proof}

\section{Proof of Corollary \ref{coro7}}

\noindent

Before proving  Corollary \ref{coro7}, we first give the following lemmas.

\begin{lemma}\label{lem7.1}
Let $p>{1}$. Then for any $f\in{L^{p}(\Omega)}$, there exists a positive constant $C$ such that
$$
\|{M}_{\varphi,\Omega}{f}\|_{L^{p}(\Omega)}\leq{C\|{f}\|_{L^{p}(\Omega)}}.
$$
\end{lemma}

\begin{proof}
Since $\varphi(|B(x,r)|)\geq{1}$, then ${M}_{\varphi,\Omega}{f}(x)\leq{{M}_{\Omega}{f}(x)}$ a.e. $x\in{\Omega}$. Thus, we can derive the conclusion of Lemma \ref{lem7.1} by the fact that $\|{M}_{\Omega}{f}\|_{L^{p}(\Omega)}\leq{C\|{f}\|_{L^{p}(\Omega)}}$ in \cite[Lemma 3.2]{HLX2}.
\end{proof}

Lemma \ref{lem7.1} together with an argument similar to that used in the proof of Lemma \ref{lem3.1} implies that the following lemma is true.

\begin{lemma}\label{lem7.2}
Suppose that $f\in{W^{1,p}(\Omega)}$ with $1<p<{\infty}$ and $0<l<1$. Then $|\nabla{K_{\varphi}^{l}f}|\in{L^{p}(\Omega)}$. Moreover, for almost every $x\in{\Omega}$, there holds
$$
|\nabla{K_{\varphi}^{l}f}(x)|
\leq{\gamma{n}{M}_{\varphi,\Omega}{f}(x)
+2{M}_{\varphi,\Omega}{|\nabla{f}|}(x)}.
$$
\end{lemma}

\begin{proof}[The proof of Corollary \ref{coro7}]
The proof of $(i)$ is analogous to the proof of Theorem \ref{thm0}, but with Lemma \ref{lem2.1} and Lemma \ref{lem3.1} replaced by Lemma \ref{lem7.1} and Lemma \ref{lem7.2}, respectively. The details are omitted here.

Now, we prove $(ii)$ is correct. By Lemma \ref{lem7.1} and Corollary \ref{coro7} $(i)$, we have
\begin{eqnarray*}
\begin{split}
\|{M}_{\varphi,\Omega}{f}\|_{W^{1,p}(\Omega)}
&\leq{\|{M}_{\varphi,\Omega}{f}\|_{L^{p}(\Omega)}
+\gamma{n}\|{M}_{\varphi,\Omega}{f}\|_{L^{p}(\Omega)}
+2\|{M}_{\varphi,\Omega}{|\nabla{f}|}\|_{L^{p}(\Omega)}}\\
&\leq{C\|f\|_{L^{p}(\Omega)}
+C\|f\|_{L^{{p}}(\Omega)}
+C\|\nabla{f}\|_{L^{p}(\Omega)}}\\
&\leq{C\|f\|_{W^{1,p}(\Omega)}}.
\end{split}
\end{eqnarray*}

Finally, we prove $(iii)$ holds. For $f\in{W_{0}^{1,p}(\Omega)}$, there is a sequence $\{\psi_{j}\}_{j=1}^{\infty}\in{W^{1,p}(\Omega)\cap{\mathcal{C}_{0}^{\infty}(\Omega)}}$ such that $\psi_{j}\rightarrow{f}$ on $W^{1,p}(\Omega)$ as $j\rightarrow{\infty}$, then it follows from Corollary \ref{coro7} $(ii)$ that $\{{M}_{\varphi,\Omega}{\psi_{j}}\}_{j=1}^{\infty}\in{W^{1,p}(\Omega)}$. If $2\mathrm{dist}(x,\Omega^{c})<\mathrm{dist}(\mathrm{supp}(\psi_{j}),\Omega^{c})$, then we can easily get that ${M}_{\varphi,\Omega}{\psi_{j}}(x)=0$. Thus, $\{{M}_{\varphi,\Omega}{\psi_{j}}\}_{j=1}^{\infty}\in{W_{0}^{1,p}(\Omega)}$ is a bounded sequence. Applying the sublinearity of maximal operators and Lemma \ref{lem7.1}, we obtain
$$
\|{M}_{\varphi,\Omega}{\psi_{j}}-{M}_{\varphi,\Omega}f\|_{L^{p}(\Omega)}
\leq{\|{M}_{\varphi,\Omega}({\psi_{j}}-f)\|_{L^{p}(\Omega)}}
\leq{C\|{\psi_{j}}-f\|_{L^{{p}}(\Omega)}}
\leq{C\|{\psi_{j}}-f\|_{W^{1,p}(\Omega)}}.
$$
Therefore, ${M}_{\varphi,\Omega}{\psi_{j}}\rightarrow{M}_{\varphi,\Omega}f$ on $L^{p}(\Omega)$ as $j\rightarrow{\infty}$. By the weak compactness argument, it holds that ${M}_{\varphi,\Omega}f\in{W_{0}^{1,p}(\Omega)}$. The proof is complete.
\end{proof}

\subsection*{Acknowledgements} This research is supported by the National Natural Science Foundation of China(Grant No. 12361018).

\end{document}